\documentclass[10pt]{amsart}
\usepackage{graphicx, mathabx, amssymb,amsfonts,amsmath,amsthm,newlfont}
\usepackage{epsfig,url}
\usepackage[usenames,dvipsnames]{color}
\usepackage{enumerate}
\usepackage[colorlinks=true,linkcolor=red,citecolor=blue]{hyperref}
\usepackage{color}
\usepackage{young}

\usepackage[all,2cell]{xy} \UseAllTwocells \SilentMatrices

\newtheorem{thm}{Theorem}[section]

\newtheorem{lem}[thm]{Lemma}

\theoremstyle{definition}
\newtheorem{defn}[thm]{Definition}
\newtheorem{notation}[thm]{Notation}

\newtheorem{example}[thm]{Example}

\theoremstyle{remark}
\newtheorem{rem}[thm]{Remark}

\numberwithin{equation}{section}

\newcommand{\To}{\longrightarrow}

\newcommand{\Z}{\mathbb Z}
\newcommand{\Q}{\mathbb Q}

\newcommand{\C}{\mathbb C}

\newcommand{\z}{\underline{z}}

\newcommand{\lf}{\langle}
\newcommand{\rf}{\rangle}

\newcommand{\bwedge}{\mathsf{\Lambda}}
\newcommand{\Sym}{\mathrm{Sym}}

\begin{document}

\title{Multivariable Vandermonde determinants, amalgams of matrices  and  Specht modules}
\author{Francis Brown}
\maketitle

\begin{abstract} Using results of Fayers on the structure of Specht modules, we prove two different formulae for the determinant of   matrices which are obtained by amalgamating the entries of two smaller matrices. In particular,  this gives formulae for multivariable Vandermonde determinants as a sum of completely factorising terms, each of which is a Vandermonde determinant in fewer variables.  As an application, we deduce  an elementary proof of the multiplicativity of the transfinite diameter for products of compact sets. 
\end{abstract}

\section{Introduction}
It is a very well-known and  often used fact that the determinant of the  classical Vandermonde matrix completely factorises into linear factors. This, however, fails  for Vandermonde determinants in several variables. In this paper, we prove two formulae for  multivariable Vandermonde determinants  as a \emph{sum  of terms}   which factorise into a product of Vandermonde determinants in fewer variables.  Each such term  separates variables into distinct factors,  and, in some cases, completely factorises  into linear terms.

Our formulae are deduced from  a more general determinantal identity for an `amalgamated product' of matrices, which is obtained by taking the Kronecker tensor products of their rows.  The proof uses  Schur-Weyl duality and the representation theory of the symmetric group. 

The identites obtained in this paper have  possible applications  to  potential theory, to multivariate interpolation in approximation  theory, and to classical problems in algebraic geometry.
In particular, we use these identities to give a simple proof of the multiplicativity of the transfinite diameter.

\section{Multivariable Vandermonde matrices} \label{sect: MultivariableVDM}

\subsection{Definition}
Let  $r\geq 1$ and let $\underline{N}=(N_1,\ldots, N_r)$ denote an ordered  sequence of integers $N_i \geq 2$. Set $m= N_1N_2 \ldots N_r$.
 Consider the ordered set 
 \[ E_N= ( e_1(\z) \ , \  e_2(\z) \ , \  \ldots  \ , \  e_{m}(\z) ) \]
  of monomials in $r$ variables $\z=(z^{(1)},\ldots, z^{(r)})$, which are of degree $\leq N_i-1$ in each component $z^{(i)}$, and  ordered degree-lexicographically. By this we mean that they are to be ordered first with respect to the total degree, and then within each degree by the lexicographic order.  
  Given $m$ such vectors $\z_1,\ldots, \z_m$, each of length $r$,  we shall define the  multivariable Vandermonde matrix  to be
\begin{equation}  \label{defn: VNmatrix}  V_{\underline{N}} (\z_1,\ldots, \z_m) =  (e_j(\z_i) )\ . \end{equation}
It is a square $m\times m$ matrix.  Note that this is one of several possible definitions for a multivariable Vandermonde matrix: some authors replace the set $E_N$ with the set of all monomials of bounded total degree (see \S\ref{sect: Homogeneous}).

\begin{example}When $r=1$, we may write $\z_i= (z_i)$ and  $E_N= (1, z, z^2, \ldots, z^{N-1})$. We retrieve the classical Vandermonde matrix
\[ V_{\underline{N}}  ( z_1,\ldots, z_N) = \begin{pmatrix} 
1 & z_1 & z_1^2 &   \ldots & z_1^{N-1}  \\
1 & z_2 & z_2^2 &   \ldots & z_2^{N-1}  \\
\vdots  & \vdots&   \vdots &   \ddots &  \vdots \\
1 & z_N & z_N^2 &   \ldots & z_N^{N-1}  
\end{pmatrix} \ . \]
As is  well known, its determinant is given by the formula
\begin{equation}  \label{detVandermonde1dimension} 
\det V_N(z_1,\ldots, z_N) =   \prod_{1\leq i<j\leq N}  (z_j - z_i) \ . 
\end{equation} 
This can be   deduced from the fact that the determinant vanishes when any two rows are equal, i.e., when  $z_i=z_j$. 
\end{example} 
The literature laments the fact that the determinant of the multivariable version of the Vandermonde matrix does not factorise, which is perfectly true. However, one notices that its determinant certainly vanishes when $\z_i=\z_j$ agree in all components, which defines  a condition of codimension $r$.  This fact suggests that there must exist  some generalisation of 
\eqref{detVandermonde1dimension}, and indeed, in this note we shall prove two such formulae which express the determinant $\det V_{\underline{N}}(\z_1,\ldots, \z_m)$ in terms of determinants involving strictly fewer variables.

\begin{example} \label{intro: example22}
Consider a simple case when $r=2$.  Let  $\underline{N}=(2,2)$, $m=4$  and write $\z_i= (z_i^{(1)}, z_i^{(2)})$. Then $E_{\underline{N}} = ( 1, z^{(1)}, z^{(2)}, z^{(1)} z^{(2)})$ and  
\[V_{(2,2)}(\z_1,\ldots \z_4)= \begin{pmatrix} 1 & z^{(1)}_1 & z^{(2)}_1 & z^{(1)}_1 z^{(2)}_1 \nonumber\\ 
1 & z^{(1)}_2 & z^{(2)}_2 & z^{(1)}_2 z^{(2)}_2  \nonumber\\ 
1 & z^{(1)}_3 & z^{(2)}_3 & z^{(1)}_3 z^{(2)}_3  \nonumber\\ 
1 & z^{(1)}_4 & z^{(2)}_4 & z^{(1)}_4 z^{(2)}_4  \nonumber\\ 
\end{pmatrix}  \  .
\]
In this case, our first formula states that: 
\begin{eqnarray}  \det V_{(2,2)}(\z_1,\ldots, \z_4)  &=&   \det V_1(z^{(1)}_1,z^{(1)}_2) \det V_1(z^{(1)}_3,z^{(1)}_4)  \nonumber  \\
&& \qquad \qquad \qquad \qquad \qquad \times \,   \det V_1(z^{(2)}_1,z^{(2)}_3)  \det V_1(z^{(2)}_2,z^{(2)}_4) \nonumber \\
&&\quad  - \quad     \det V_1(z^{(1)}_1,z^{(1)}_3) \det V_1(z^{(1)}_2,z^{(1)}_4)  \nonumber \\
&& \qquad \qquad \qquad \qquad \qquad \times \,    \det V_1(z^{(2)}_1,z^{(2)}_2) \det V_1(z^{(2)}_3,z^{(2)}_4)  \nonumber\\
& = & \det \begin{pmatrix} 1   &  z^{(1)}_1  \\ 1 & z^{(1)}_2  \end{pmatrix}\det \begin{pmatrix} 1   &  z^{(1)}_3  \\ 1 & z^{(1)}_4  \end{pmatrix}   \det \begin{pmatrix} 1   &  z^{(2)}_1  \\ 1 & z^{(2)}_3  \end{pmatrix} \det \begin{pmatrix} 1   &  z^{(2)}_2  \\ 1 & z^{(2)}_4  \end{pmatrix} \nonumber \\ 
 &&  \quad  - \quad      \det \begin{pmatrix} 1   &  z^{(1)}_1  \\ 1 & z^{(1)}_3  \end{pmatrix}\det \begin{pmatrix} 1   &  z^{(1)}_2  \\ 1 & z^{(1)}_4  \end{pmatrix}  \det \begin{pmatrix} 1   &  z^{(2)}_1  \\ 1 & z^{(2)}_2  \end{pmatrix} \det \begin{pmatrix} 1   &  z^{(2)}_3  \\ 1 & z^{(2)}_4  \end{pmatrix} \nonumber .
 \end{eqnarray} 
 For the convenience of the reader, we may write this in more readable form by relabelling the variables, giving the identity: 
 \begin{multline} \det  \begin{pmatrix} 1 & x_1 & y_1 & x_1 y_1 \nonumber\\ 
1 & x_2 & y_2 & x_2 y_2  \nonumber\\ 
1 & x_3 & y_3 & x_3 y_3  \nonumber\\ 
1 & x_4 & y_4 & x_4 y_4  \nonumber\\ 
\end{pmatrix} 
=     (x_2-x_1)  (x_4-x_3)  (y_3-y_1)  (y_4-y_2)   \\
 -   (x_3-x_1)  (x_4-x_2)  (y_2-y_1)  (y_4-y_3)    \ . 
 \end{multline}
One may check that, indeed, if any two rows are equal then it vanishes. 
See example \ref{example:V32} for the case $N=(3,2)$.

The second formula which we prove also expresses $V_N$ as a linear combination of completely factorising terms in separate sets of variables. It  is more symmetric than the first,  but involves a greater number of terms.  In this case it states that 
\begin{eqnarray}  \det V_{(2,2)}  &= & \frac{1}{12} \sum_{\sigma \in \Sigma_4}  \varepsilon(\sigma)  \, \det V_1(z^{(1)}_{\sigma_1},z^{(1)}_{\sigma_2}) \det  V_1(z^{(1)}_{\sigma_3},z^{(1)}_{\sigma_4}) \nonumber \\
&&  \qquad \qquad \qquad \qquad \qquad \times \,  \det V_1(z^{(2)}_{\sigma_1},z^{(2)}_{\sigma_3}) \det V_1(z^{(2)}_{\sigma_2},z^{(2)}_{\sigma_4}) \nonumber \\
& = &    \frac{1}{12} \sum_{\sigma \in \Sigma_4}  \varepsilon(\sigma)  \, (z^{(1)}_{\sigma_2} -  z^{(1)}_{\sigma_1}) (z^{(1)}_{\sigma_4}- z^{(1)}_{\sigma_3}) \times  (z^{(2)}_{\sigma_3}- z^{(2)}_{\sigma_1}) (z^{(2)}_{\sigma_4}- z^{(2)}_{\sigma_2})   \nonumber  \end{eqnarray}   
 where the symmetric group $\Sigma_4$ acts on the four subscripts $\{1,2,3,4\}$ of the $\z_i$, and $\varepsilon(\sigma)$ is the sign of a permutation $\sigma$.

\end{example}

\subsection{First formula  for Vandermonde determinants} \label{sect: firstformula} 
We may express the determinant of the Vandermonde matrix \eqref{defn: VNmatrix} in $r$ variables as a sum of terms  which factorise into a product of Vandermonde determinants in $k$ and $r-k$ variables. 

 Let $N_1,\ldots, N_r\geq 2$ and let  $1\leq k< r$. 
Let 
\[ m= N_1\ldots N_k  \quad  \hbox{ and  }\quad  n= N_{k+1} \ldots N_r\ \]
so that $mn = N_1\ldots N_r$ is the rank of $V_{(N_1,\ldots, N_r)}$.
Consider $mn$ vectors  $\z_1,\ldots, \z_{mn}$ of  length $r$. 
 The coordinates of each variable $\z_i$  may be partitioned into two vectors
 \[  \z_i = (\z'_i, \z_i'') 
 \]
of length $k$ and $r-k$ respectively, where 
\[ \z'_i = (z_i^{(1)}, \ldots, z_i^{(k)}) \quad \hbox{ and } \quad \z''_i = (z_i^{(k+1)}, \ldots, z_i^{(r)})\ .
\]
We shall express $V_{N_1,\ldots, N_r}(\z_1,\ldots, \z_{mn})$ in terms of products of  $V_{N_1,\ldots,N_k}(\z'_i, i\in I)$ and 
  $V_{N_{k+1},\ldots,N_r}(\z''_j, j\in J)$, where $I$ and $J$ range over subsets of $\{1,\ldots, mn\}$.

In order to state the formula, recall that a standard Young tableau of shape $n^m=(n,\ldots,n)$ is a square array with $m$ rows and $n$ columns  of boxes in which the numbers $1,\ldots, mn$ are placed in such a way that every row and column is increasing when read from left to right, and top to bottom, respectively. For instance, the trivial tableau is as follows:
\vspace{0.05in}
\[ 
1_{n^m} = \begin{array}{|c|c|c|c|c|} \hline
 1 &  m+1 & \ldots & m(n-2)+1 &  m(n-1)+1    \\  \hline
 2 & m+2  & \ldots &   m (n-2)+2 &   m(n-1)+2 \\ \hline
\vdots   &   \vdots  &    &  \vdots& \vdots    \\ \hline
m  &  2m & \ldots        & m(n-1) & mn  \\ \hline
\end{array}
\] 
\vspace{0.05in}

Our first  formula for $\det V_{\underline{N}}$ is not completely explicit but is  relatively parsimonious. It involves summing over pairs of  standard tableaux. 
\begin{thm} \label{thm: FirstVDMformula} The multivariable Vandermonde determinant satisfies:
\begin{multline}  \det V_{(N_1,\ldots, N_r)} (\z_1,\ldots , \z_{mn}) =   \label{Intro:FirstVDMformula} \\
 \sum_{\alpha,\beta}  \Phi_{\alpha,\beta} \prod_{i=1}^n  \det V_{(N_1,\ldots, N_k)} (z'_k, k\in \mathrm{col}_i(\alpha))  \prod_{j=1}^m  \det V_{(N_{k+1},\ldots, N_r)} (z''_k, k\in \mathrm{row}_j(\beta))   \end{multline}
 where the sum is over all pairs of standard tableaux $\alpha,\beta$ of shape $n^m$, and the $\Phi_{\alpha,\beta}$ are certain integers. The notation $\mathrm{col}_i(\alpha)$ denotes the set of elements in the $i^{\mathrm{th}}$ column of $\alpha$, in ascending order, and likewise $\mathrm{row}_j(\beta)$ denotes the $j^{\mathrm{th}}$ row of $\beta$.
 \end{thm} 

This formula depends on the choice of ordering of the rows $\z_1,\ldots, \z_{mn}$ since  the notion of a standard tableau also does. The integers $\Phi_{\alpha,\beta}$ are the entries of   a matrix $\Phi$ which is  the inverse conjugate of a matrix introduced by Fayers. It is a sparse and explicit quasi-unipotent  matrix whose entries are $0,\pm 1$ (although its inverse  conjugate $\Phi$ can have entries which are larger integers).

\begin{example} The two standard tableaux of shape $(2,2)$ are: 
\[  \underbrace{\begin{Young}
1 & 3  \cr
2  & 4 \cr
\end{Young}}_{1_{2,2}}   \quad   , \quad   \underbrace{\begin{Young}
1 & 2  \cr
3  & 4  \cr
\end{Young}}_{1'_{2,2}}   \ .  \] 
The matrix $\Phi$ with respect to  $1_{2,2}, 1'_{2,2}$ is 
 \[ \Phi = \begin{pmatrix}  1 & 0 \\ 0 & -1 \end{pmatrix} \ . \]
 Formula \eqref{Intro:FirstVDMformula} then reduces to the formula in example \ref{intro: example22}.
\end{example}

\subsection{Second formula for the determinant}
Our second formula is in completely closed form, but involves summing over a large number of terms.
\begin{thm} \label{thm: SecondFormula} Without loss of generality, let $m\geq n$. Let $\Sigma_{mn}$ denote the symmetric group which acts by permutations of $\{1,\ldots, mn\}$.  For any $\sigma \in \Sigma_{mn}$,  let ${}^\sigma1_{n^m}$ denote the (non-standard) tableau obtained from $1_{n^m}$ by permuting its entries. Let 
\[ H_{m,n} =  \prod_{i=0}^{n-1}  \frac{(m+i)!}{i!}  \ .\]
The  multivariable Vandermonde determinant satisfies:
\begin{multline}  \det V_{(N_1,\ldots, N_r)} (\z_1,\ldots , \z_{mn}) =  \frac{1}{H_{m,n}}  \sum_{\sigma \in \Sigma_{mn}}  \varepsilon(\sigma)    \label{Intro:SecondVDMformula} \\
  \prod_{i=1}^n  \det V_{(N_1,\ldots, N_k)} (z'_k, k\in \mathrm{col}_i({}^\sigma1_{n^m}) )  \prod_{j=1}^m  \det V_{(N_{k+1},\ldots, N_r)} (z''_k, k\in \mathrm{row}_j({}^\sigma1_{n^m}))     \ . \end{multline}
  where $z'_k, k\in \mathrm{row}_i({}^\sigma1_{n^m}) $ denotes the $z'_k$ where the indices $k$ are taken from the elements in the ith row of ${}^\sigma1_{n^m}$ in  order (from left to right), and similarly, $\mathrm{col}_j({}^\sigma1_{n^m})$ denotes the elements in the jth column of ${}^\sigma1_{n^m}$, ordered from top to bottom.
\end{thm}
The number  $H_{m,n}$ is the  product of the  Hook lengths for a tableau of shape $m^n$ or $n^m$.

\begin{example} Let $r=2$ and $m=3, n=2$. We have
\[ 
1_{2^3} = \begin{array}{|c|c|c|} \hline
 1 &  4    \\  \hline
 2 & 5    \\ \hline
3 & 6   \\ \hline
\end{array}
\] 
 and therefore theorem \ref{thm: SecondFormula} states that
\begin{multline}  \det V_{(3,2)}(\z_1,\ldots, \z_6)  =\\  
 \frac{1}{144}  \sum_{\sigma \in \Sigma_6} \varepsilon(\sigma)   \det   V_{3}(z'_{\sigma_1},z'_{\sigma_2},z'_{\sigma_3}) \det V_{3}(z'_{\sigma_4},z'_{\sigma_5},z'_{\sigma_6}) \\
 \times \,   \det V_{3}(z''_{\sigma_1},z''_{\sigma_4})\det V_{3}(z''_{\sigma_2},z''_{\sigma_5})\det V_{3}(z''_{\sigma_3},z''_{\sigma_6})  \end{multline}
When $r=2$ and $m=n=3$, we have
\[ 
1_{3^3} = \begin{array}{|c|c|c|} \hline
 1 &  4 & 7   \\  \hline
 2 & 5  & 8  \\ \hline
3 & 6 & 9  \\ \hline
\end{array}
\] 
 and theorem \ref{thm: SecondFormula} yields
\begin{multline}  \det V_{(3,3)}(\z_1,\ldots, \z_9)  =\\  
 \frac{1}{8640}  \sum_{\sigma \in \Sigma_9} \varepsilon(\sigma)   \det  V_{3}(z'_{\sigma_1},z'_{\sigma_2},z'_{\sigma_3}) \det V_{3}(z'_{\sigma_4},z'_{\sigma_5},z'_{\sigma_6}) \det V_{3}(z'_{\sigma_7},z'_{\sigma_8},z'_{\sigma_9}) \\
 \times   \, \det V_{3}(z''_{\sigma_1},z''_{\sigma_4},z''_{\sigma_7}) \det V_{3}(z''_{\sigma_2},z''_{\sigma_5},z''_{\sigma_8})\det V_{3}(z''_{\sigma_3},z''_{\sigma_6},z''_{\sigma_9})  \end{multline}
\end{example}

These formulae are proved as special cases of a more general determinantal identity for matrices, which will be postponed to section \S\ref{sect: IdentityAstarB}.

\begin{rem}
By repeated application of either  formula \eqref{Intro:FirstVDMformula} or \eqref{Intro:SecondVDMformula} to reduce to the one-variable case,  and equation \eqref{detVandermonde1dimension},  we deduce that the general multivariable Vandermonde determinant is a sum of completely factorising terms in which the variables are  separated. 
\end{rem}

\subsection{`Homogeneous' Vandermonde determinants and interpolation} \label{sect: Homogeneous}
A common version of multivariable Vandermonde determinants involves only considering the union of the set of  monomials which are \emph{homogeneous}  of given degrees $0,\ldots,N$.

To define them, let $N>0$  be an integer and $r\geq 1$. Let 
$E^{\mathrm{hom}}_N=(e_j(\z))$ denote the ordered set of monomials in
 $r$ variables $\z= (z^{(1)}, \ldots, z^{(r)})$  which are homogeneous of  degree  $k$, for all $k\leq N$. They are to be ordered first with respect to their total degree (from smallest to largest), and then within each   degree, they are to be ordered lexicographically. Let $\ell$ denote the number of such monomials. The `homogenous' Vandermonde determinant is the square matrix:
\begin{equation} \label{HomVandermonde}
V_{N,r}^{\mathrm{hom}} (\z_1,\ldots, \z_\ell) = ( e_j(z_i)) \ .
\end{equation}

\begin{example} \label{example: Pascal}
Let $p_i=(x_i,y_i)$ for $i=1,\ldots, 6$ and consider the   matrix
\begin{equation} \label{Vhomconic} V^{\mathrm{hom}}_{2,2}(p_1,\ldots, p_6) =  \begin{pmatrix} 1 & x_1 & y_1 & x_1^2 & x_1y_1 & y_1^2  \\ 
1 & x_2 & y_2 & x_2^2 & x_2y_2 & y_2^2  \\
1 & x_3 & y_3 & x_3^2 & x_3y_3 & y_3^2  \\
1 & x_4 & y_4 & x_4^2 & x_4y_4 & y_4^2  \\
1 & x_5 & y_5 & x_5^2 & x_5y_5 & y_5^2  \\
1 & x_6 & y_6 & x_6^2 & x_6y_6 & y_6^2  \\
\end{pmatrix}
\end{equation}
Its determinant vanishes if and only if $p_1,\ldots, p_6$ lie on a conic. A classical problem from antiquity is to determine  geometric conditions on the $p_i$ for this to hold.
 More precisely, if $\underline{c}=(c_1,\ldots, c_6)^T$   is in the kernel of  $V^{\mathrm{hom}}_{2,2}$, then the points $p_i$ lie on the conic $f(x,y)=0$, where $f(x,y)=c_1 + c_2 x + c_3 y +c_4 x^2 + c_5 xy + c_6 y^2$.
\end{example}

Although such matrices  are not our primary focus, one observes that  $V_{N,r}^{\mathrm{hom}} $ may always be  obtained as a submatrix of  $V^{(N,\ldots, N)}$ by deleting rows and columns. It follows that  the proofs of the two previous theorems provide analogous statements for  homogeneous Vandermonde determinants.

\begin{thm}  \label{thm: homogeneous} The determinant $ \det(V_{N,r}^{\mathrm{hom}} )(\z_1,\ldots , \z_{\ell})$  of the homogeneous  Vandermonde matrix  may be written as an explicit sum of products of  degenerate Vandermonde matrices in fewer variables in which the variables separate. \end{thm}

This may be done in two different ways,  by taking degenerate  versions of the expressions in either  theorem \ref{thm: FirstVDMformula} or \ref{thm: SecondFormula} as we now illustrate.

\begin{example} \label{Ex: Pascalseparate}
Consider  example \ref{example: Pascal}.
Without loss of generality, we may assume that $x_2=x_1$, i.e., points $p_1$ and $p_2$ have the same $x$ coordinate. Then theorem \ref{thm: detAB} produces the  following expression for 
$
\det  V^{\mathrm{hom}}_{2,2}(p_1,\ldots, p_6)  $: 
\begin{eqnarray} 
& \Big( \!\!\!\!\!\!&  V_{1,3,4}(x)\, (x_1-x_6) \times (y_2-y_5)(y_1-y_5)(y_3-y_6)   \nonumber \\ 
&&  V_{1,3,5}(x)\,  (x_1-x_4) \times (y_2-y_6)(y_1-y_6)(y_3-y_4)\nonumber \\ 
&&  V_{1,3,6}(x) \,  (x_1-x_5) \times (y_2-y_4)(y_1-y_4)(y_3-y_5) \nonumber\\ 
  && V_{1,4,5}(x) \, (x_1-x_6) \times (y_2-y_3)(y_1-y_3)(y_4-y_6)  \nonumber\\
  && - V_{1,3,4}(x) \,   (x_1-x_5) \times (y_2-y_6)(y_1-y_6)(y_3-y_5) \nonumber\\ 
   && - V_{1,3,5}(x) \,    (x_1-x_6)\times (y_2-y_4)(y_1-y_4)(y_3-y_6) \nonumber\\
  &&  - V_{1,3,6}(x) \, (x_1-x_4) \times (y_2-y_5)(y_1-y_5)(y_3-y_4) \nonumber\\
  && -  V_{1,4,6}(x) \,   (x_1-x_5)\times (y_2-y_3)(y_1-y_3)(y_4-y_5)
   \Big) \times (y_1-y_2) \nonumber \ ,
    \end{eqnarray}
    where $V_{i,j,k}(x) = (x_i-x_j)(x_j-x_k)(x_i-x_k)$. 
    Each line is thus a  product of Vandermonde determinants in a single variable. 
    This expression further collapses under additional assumptions on the location of the points $p_i$, which we leave for the amusement of the reader.
   \end{example}

In fact, our technique also provides expressions for this determinant as a sum of factorising terms in which the variables \emph{do not separate}, but which  encode information about collinearity of the $p_i$. We illustrate again using  example \ref{example: Pascal}.

\begin{example} 
Without loss of generality, we may assume $p_6 =(0,0)$ is  the origin and $p_5$ lies on the $x$-axis (i.e.,  $y_5=0$). Then theorem \ref{thm: detAB} also
 delivers a sum of products  of   degenerate  homogeneous Vandermonde determinants of degree $1$:
\begin{multline}  \det\, V^{\mathrm{hom}}_{2,2} =    \det  \begin{pmatrix} 1 & x_1 & y_1\\
1& x_2 & y_2\\
1& x_3 & y_3 \end{pmatrix} (x_1 y_4 - x_4 y_1) (x_2 y_5 -x_5 y_2) y_3 (x_5-x_4) \\
\qquad \qquad  \qquad  +  \det  \begin{pmatrix} 1 & x_1 & y_1\\
1& x_3 & y_3\\
1& x_4 & y_4 \end{pmatrix}  (x_1 y_2 - x_2 y_1) (x_3 y_5 -x_5 y_3)  y_4 (x_5-x_2)  \\
   -    \det  \begin{pmatrix} 1 & x_1 & y_1\\
1& x_2 & y_2\\
1& x_4 & y_4 \end{pmatrix}   \,  (x_1 y_3 - x_3 y_1) (x_2 y_5 -x_5 y_2)   y_4 (x_5-x_3) \ .   
   \end{multline}
   Each one of the  $3\times 3$ determinants in the previous expression vanishes whenever the three points corresponding to its entries are collinear.
   \end{example}

\subsection{Applications}
The formulae in this paper may be used to compute determinants of matrices where the entries in every row satisfy multiplicative relations uniformly in each column, i.e., each entry is a polynomial in a smaller number of matrix entries. 
The main statements are 
theorems \ref{thm: detAB} and \ref{thm: detABSymmetric} which provide a formula for an `amalgamated product' of two matrices and are proved using the representation theory of the symmetric group. This might be a useful complement to the  methods for computing determinants described in \cite{Krattenthaler,Krattenthaler2}.  As remarked above, these formulae can be applied in different ways, in both generic and degenerate form, to compute multivariable Vandermonde determinants of both rectangular and homogeneous types.

Vandermonde determinants and their variants appear in a variety of different fields including approximation theory and multivariate interpolation \cite{Olver}, coding theory \cite{ToricVdM}, algebraic geometry \cite{TropicalVdM},  harmonic analysis \cite{ClusteredNodesVdM} and potential theory \cite{BloomLevenberg}.  We provide just one application  of our formulae  to give a simple proof of the multiplicativity of the transfinite diameter \S\ref{sect: applicationtransfinite}. 
We  expect that our formulae could also be used to explain factorisation and vanishing properties of generalised Vandermonde determinants for specific configurations  of points.  See \cite{SeparateVdM} and references therein. Further examples are considered in \ref{example:V32} and \ref{example: Pascalconic}.  

\section{A determinant identity for an amalgam of two matrices} \label{sect: IdentityAstarB} 
We derive a formula for the determinant of a matrix obtained from two smaller matrices by taking the Kronecker tensor products of each of their rows.

\subsection{An amalgamated matrix}
Let $m,n\geq 1$ and consider a  matrix with $mn$ rows and $m$ columns:
\begin{equation} \label{matrixAdef} A=   \begin{pmatrix}  a_{11} & a_{12} &  \ldots & a_{1m} \\ 
a_{21} & a_{22} &  \ldots & a_{2m} \\ 
 \vdots    &   \vdots&  & \vdots \\
a_{mn\, 1 } & a_{mn \, 2} &  \ldots & a_{mn \, m} \\ 
 \end{pmatrix}  = 
 \begin{pmatrix}  \underline{a}_1 \\ 
 \underline{a}_2 \\ 
      \vdots \\
\underline{a}_{mn} \\ 
 \end{pmatrix}    
 \end{equation} 
 with entries $a_{i  j}$ in a commutative ring $R$, and where $\underline{a}_i = (a_{i 1} \ldots a_{i  m} ) \in R^m$ denotes the $i^{\mathrm{th}}$ row vector.
 Likewise, consider the matrix with  $mn$ rows and $n$ columns 
\begin{equation} \label{matrixBdef}    B=   \begin{pmatrix}  b_{11} & b_{12} &  \ldots & b_{1 n } \\ 
b_{21} & b_{22} &  \ldots & b_{2 n } \\ 
 \vdots    &  \vdots &  & \vdots \\
b_{mn\, 1 } & b_{mn \, 2} &  \ldots & b_{mn \,  n } \\ 
 \end{pmatrix} =  \begin{pmatrix}   \underline{b}_1 \\ 
 \underline{b}_2 \\ 
      \vdots \\
\underline{b}_{mn} \\ 
 \end{pmatrix}     \end{equation}
 with entries $b_{ij}\in R$, and  $\underline{b}_i = (b_{i  1} \ldots b_{i  n} ) \in R^n$ denotes the $i^{\mathrm{th}}$ row.
  
  \begin{defn}  Consider  a new matrix which   amalgamates  these two matrices,  by taking the Kronecker tensor product of their  rows:
 \[  
 A \star B =     \begin{pmatrix} \underline{a}_1 \otimes   \underline{b}_1 \\ 
  \underline{a}_2 \otimes \underline{b}_2 \\ 
      \vdots \\
 \underline{a}_{mn} \otimes  \underline{b}_{mn} \\ 
 \end{pmatrix}  \]
 It is a square matrix with  $mn$ rows and columns, and entries in $R$.  
 \end{defn} Our convention for the Kronecker tensor product of two row vectors is:
 \begin{multline}  \label{KroneckerTensorConvention} (a_1  ,  \ldots ,  a_m) \otimes (b_1,   \ldots  , b_n)  =   \\ 
   ( \underbrace{a_1 b_1  ,  a_2 b_1  ,  \ldots ,   a_m b_1}_{\underline{a} \otimes b_1} , \underbrace{a_1 b_2 , a_2 b_2 , \ldots, a_m b_2}_{\underline{a} \otimes b_2}  , \ldots ,   \underbrace{a_1 b_n , a_2 b_n , \ldots, a_m b_n}_{\underline{a} \otimes b_n}   )  \end{multline} 
 where commas between entries have been inserted to assist with  legibility and $\underline{a} = (a_1,\ldots, a_m)$.  This convention must of course be applied consistently to every row in the definition of the matrix $A \star B$.  \subsection{Special case: Kronecker tensor product}
 Suppose that $C$ is a square $m\times m$ matrix, and $D$ is a square $n\times n$ matrix with rows $\underline{d}_1,\ldots, \underline{d}_n$. Then the Kronecker tensor product $C\otimes D$ may be represented by the matrix $A \star B$ where
 \begin{equation} \label{KroneckerCase} A =  \begin{pmatrix}    C  \\  C  \\  \vdots\\  C \end{pmatrix} \in M_{mn,m}(R)   \quad \hbox{ and } \quad B =   \begin{pmatrix}    m\begin{cases} \underline{d}_1\\ \vdots\\ \underline{d}_1 \end{cases} \\ \vdots \\ m\begin{cases} \underline{d}_n\\ \vdots\\ \underline{d}_n \end{cases}   \\  \end{pmatrix}\in M_{mn,n}(R) 
\end{equation}
It is well-known that $\det(C \otimes D) = \det(C)^n \det(D)^m$. 
 
 \subsection{Minors}
 If $I \subset \{1,\ldots, mn\}$ is any subset with $m$ elements, let $A_I$ denote the square matrix with $m$ rows and columns given as follows:
 \[ A_I =  \begin{pmatrix}  a_{i_11} & a_{i_12} &  \ldots & a_{i_1m} \\ 
a_{i_21} & a_{i_22} &  \ldots & a_{i_2m} \\ 
 \vdots    &  \vdots &  & \vdots \\
a_{i_m\, 1 } & a_{i_m \, 2} &  \ldots & a_{i_m \, m} \\ 
 \end{pmatrix}  = 
 \begin{pmatrix}  \underline{a}_{i_1} \\ 
 \underline{a}_{i_2} \\ 
      \vdots \\
\underline{a}_{i_m} \\ 
 \end{pmatrix}    
 \]
where $I= \{i_1,\ldots, i_m\}$ and $i_1<i_2< \ldots< i_m$ are increasing.  We define $B_I$ similarly by selecting the rows of $B$ indexed by $I$.

\subsection{Young tableaux}  Our formula for the determinant of $A\star B$ involves  Young diagrams. Recall that a Young diagram is specified by a partition 
$\lambda = (n_1,\ldots, n_r)$  of $N=n_1+ \ldots + n_r$, denoted by $\lambda \vdash N$.    The integers $n_i$ satisfy   $n_1\geq \ldots \geq n_r\geq 1$, and one writes $|\lambda|$ for $N$. The corresponding Young diagram is  a left-justified array of boxes with $n_i$ boxes in the $i^\mathrm{th}$ row. A standard Young tableau is defined to be a Young diagram in which every integer from $1$ to $n_1+\ldots +n_r$ is placed in every box in such a way that  every row (and column) is increasing when read from left to right (or top to bottom).  We shall denote a standard Young tableau corresponding to $\lambda$ using Greek letters $\alpha, \beta$, etc.
The conjugate diagram $\lambda'$ to $\lambda$ is obtained by reflecting $\lambda$ along the diagonal (from top-left to bottom-right) which passes through the  top-left corner of the diagram. The conjugate of a standard Young tableau $\alpha$ is also standard, and is denoted by $\alpha'$.

Below is a picture of a standard Young tableau and its conjugate:
\[  \underbrace{\begin{Young}
1 & 3 &4 \cr
2  & 5 \cr
6 & 7 \cr
\end{Young}}_{\alpha}   \quad   , \quad   \underbrace{\begin{Young}
1 & 2 & 6 \cr
3  & 5 & 7 \cr
4 \cr
\end{Young}}_{\alpha'}  \]
The formulae for $\det(A\star B)$ only involves Young tableaux of rectangular shape, i.e., of the form $n^m=(n,\ldots, n)$ with $m$ rows consisting of $n$ boxes, or their conjugates, which are  the form $m^n=(m,\ldots, m)$ with $n$ rows of length $m$.

Let $1_{\lambda}$ denote the standard Young tableau of shape $\lambda$ where the boxes are filled consecutively firstly from top to bottom, then left to right.
For any tableau $\gamma$ of shape $\lambda$, define the sign $\varepsilon(\gamma) \in \{ -1,1\}$ to be the sign of the unique permutation  $\pi$ of $\Sigma_{|\lambda|}$ such that  $\pi(1_{\lambda}) = \gamma$. One calls two (non-standard) tableau column-equivalent if they differ only by permutating  elements in each column. The notion of row-equivalence is defined similarly.

 \begin{defn}  \cite[Definition 4.7]{Fayers}. Let $\lambda$ be a Young diagram.  Given a standard Young tableau $\alpha$ of shape $\lambda$, and  $\beta$ of conjugate shape $\lambda'$, define the \emph{Fayers number} (denoted by square brackets $[\alpha,\beta]$ in \emph{loc. cit.}):
 \begin{equation} 
    \lf \alpha, \beta \rf  \ \in \  \{-1, 0, 1\}
 \end{equation} 
 as follows. It is   zero if there are two numbers in $\{1,\ldots, |\lambda|\}$  which   lie in a common column of $\alpha$ and a common column of   $\beta$. If not, 
there exists a unique tableau $\gamma$ of shape $\lambda$ such that $\gamma$ is column-equivalent  to $\alpha$, and $\gamma'$ is  column-equivalent to $\beta$.
Define $\lf \alpha, \beta \rf$ in this case to be the sign $\varepsilon(\gamma)$.\footnote{Note that the above definition appears in \emph{loc. cit.} with `column' replaced by `row' in the discussion preceding Lemma 4.4,  but is the opposite way around in  the proof of Lemma 4.5. }
 \end{defn}  
 
 \begin{defn}  Define the \emph{Fayers matrix} $F$ to be:
 \[ (F)_{\alpha, \beta} = \lf \alpha,\beta \rf \ , \]
 where $\alpha, \beta$ range over the set of standard tableaux of shape $\lambda, \lambda'$ respectively. 
 Denote by $(\mathcal{E} F \mathcal{E})$  the matrix whose entries are $(\mathcal{E} F \mathcal{E})_{\alpha, \beta} = \varepsilon(\alpha) F_{\alpha,\beta} \varepsilon(\beta)$.  Define
  \begin{equation} 
 \Phi =  ( \mathcal{E} F\mathcal{E})^* 
 \end{equation} 
 where $*$ denotes the inverse transpose. 
  \end{defn} 
 In \cite[Lemma 4.5]{Fayers} it is shown that $F$ is invertible, and furthermore  that 
 $F_{\alpha, \beta}=0$ unless $\alpha$ dominates $\beta$.  It follows that  $F$ is lower triangular with respect to any ordering of standard tableaux which respects the dominance ordering. Recall that a shape $\lambda= (\lambda_1,\ldots, \lambda_i,\ldots )$ dominates the shape  $\mu=(\mu_1,\ldots, \mu_i,\ldots)$ (with infinitely many trailing zeros)  if $\sum_{i=1}^n \lambda_i \geq \sum_{j=1}^n 
 \mu_j$ for all $n\geq 1$ (\cite[Definition 3.2]{James}). 
 A standard  tableau $\alpha$   dominates $\beta$ if the truncation of  $\alpha$ to its sub-tableaux which contain all the entries up to $k$, for every $k$, dominates the corresponding trunction of  $\beta$, for all $k\geq 1$.

 \begin{rem} In the case of interest,  when $\lambda=n^m$ is  rectangular, we may associate to a standard tableau $\alpha$ of shape $\lambda$ the increasing sequence $r_1<r_2<\ldots< r_m$ of the sums of entries in each row. The associated lexicographic ordering is a convenient way to write a basis for the  standard tableaux $\alpha$.  For instance, if $\lf \alpha, \beta \rf$ is non-zero and $\beta \neq \alpha'$ then  $\beta'$ is strictly greater than $\alpha$ in this ordering. \end{rem}

\subsection{Statement of the  identities} 
Let $A, B$ be as above. 
\begin{notation} Given a standard tableau $\alpha$ of shape $n^m$,  and $\beta$ of shape $m^n$, let 
\[ A_{\alpha } =   \prod_{i=1}^n \det A_{\mathrm{col}_i(\alpha)} \quad \hbox{ and } \quad B_{\beta} = \prod_{j=1}^m \det B_{\mathrm{col}_j(\beta)}\]
where for a standard tableau $\gamma$,
 $\mathrm{col}_i(\gamma)$ denotes  the $i^\mathrm{th}$ column of $\gamma$.

\end{notation}

The determinant of $A\star B$ separates into a  linear combination of a product of determinants of minors of $A$, and  determinants of minors of $B$. 
Here and later on, we shall denote by $\{\alpha^{\vee}\}$ the vector space basis (over $\Q$) dual to the  vector space basis $\{\alpha\}$ indexed by standard tableau of fixed shape $n^m$, and similarly for  $\beta$ (if we view $F \in \mathrm{Hom}( V,W)$ as a linear map between two vector spaces then   $\Phi \in \mathrm{Hom}(V^{\vee}, W^{\vee})$ is a map between their duals).

\begin{thm}  \label{thm: detAB}  We have the identity 
\[ \det(A\star B) = \sum_{\alpha, \beta}   \Phi_{\alpha^{\vee},\beta^{\vee}} A_{\alpha} B_{\beta} \]
where the sum is over all standard Young tableaux $\alpha$ of shape $n^m$ and $\beta$ of conjugate shape $m^n$ 
and $\Phi_{\alpha^{\vee},\beta^{\vee}}$ denotes the corresponding entry of the matrix $\Phi$.   \end{thm}

Note that the entries of  $\Phi$, which is the inverse of a quasi-unipotent matrix with integer entries, are also integers. However,  contrary to the matrix $F$, they  are not all equal  $0, -1, 1$. For example, when $m=n=3$ a   handful of entries of $\Phi$  are equal to $-2$ or $2$. The entries  $\Phi_{\alpha_1^{\vee}, \alpha_n^{\vee}}$ may be interpreted as a signed sum over  chains of relations $\alpha_1\sim \alpha_2 \sim \ldots \sim \alpha_n$ where $\alpha \sim \beta$ if $\lf \alpha, \beta' \rf$ is non-zero. 

\begin{rem} The identity may be reinterpreted as a bilinear form. Let
\[ \underline{A} =  A_{\alpha} \quad , \quad   \underline{B} =  B_{\beta}\]
denote the column vectors whose entries are the  $A_{\alpha}$
where   $\alpha$ (respectively $\beta$) range over the set of standard Young tableaux of shape $m^n$ (resp. $n^m$). Then theorem \ref{thm: detAB}  
states that:
\[  \det(A\star B) =  \underline{A}^T  \Phi\,  \underline{B}\ .\]
\end{rem} 

The previous theorem has the advantage that it contains very few terms, but the disadvantage that it involves inversion of a matrix. 
The following theorem is more symmetric, but typically involves a far higher number of terms. 

\begin{thm}  \label{thm: detABSymmetric}  Let  $\alpha,\beta$ be  standard tableaux of shape $n^m$ and $m^n$ respectively.
Then there is an integer $\kappa_{\alpha,\beta}$ such that  
\[ \det(A\star B)  \, \kappa_{\alpha, \beta}  =    \sum_{\sigma\in \Sigma_{mn}} \varepsilon(\sigma) A_{\sigma\alpha} B_{\sigma \beta} \ . \]
The integer $\kappa_{\alpha,\beta}$ vanishes if and only if $\lf \alpha, \beta\rf =0$.

When  $\alpha=1_{n^m}$ and $\beta=1'_{m^n}$ is its conjugate, then for $m\geq n$ 
\[\kappa_{\alpha,\beta} =  H_{m,n} =   \prod_{i=0}^{n-1}  \frac{(m+i)!}{i!}  \] is the product of the Hook lengths of a tableau of shape $n^m$.
\end{thm}

\subsection{Examples}

\subsubsection{$m=n=2$}We may write explicitly 
\[ A= \begin{pmatrix}
a_{11}   &   a_{12}      \\ 
a_{21}   &   a_{22}      \\
a_{31}  &   a_{32}     \\
a_{41}   &   a_{42}    
\end{pmatrix} 
\\  \quad , \quad B= \begin{pmatrix}
b_{11}  &    b_{12}    \\ 
b_{21}  &    b_{22}     \\
b_{31}  &     b_{32}    \\
b_{41}  &     b_{42}    
\end{pmatrix} 
\]
and hence
\[ A\star B = 
\begin{pmatrix}
a_{11}b_{11}  &   a_{12} b_{11}  & a_{11}b_{12}  &   a_{12} b_{12}   \\ 
a_{21}b_{21}  &   a_{22} b_{21}  & a_{21}b_{22}  &   a_{22} b_{22}   \\
a_{31}b_{31}  &   a_{32} b_{31}  & a_{31}b_{32}  &   a_{32} b_{32}    \\
a_{41}b_{41}  &   a_{42} b_{41}  & a_{41}b_{42}  &   a_{42} b_{42}    
\end{pmatrix} 
 \]

There are exactly two standard Young tableaux of shape $(2,2)$, namely:
\[   \underbrace{\begin{Young}
1 & 3 \cr
2  & 4 \cr
\end{Young}}_{\alpha_1}    \quad   , \quad  \underbrace{\begin{Young}
1 & 2 \cr
3  & 4 \cr
\end{Young}}_{\alpha_2} \ .\]
where $\alpha_1 = 1_{(2,2)}$ has sign $1$ and $\alpha_2$ has sign $-1$. 
With respect to the bases $\alpha_1, \alpha_2$ and $\beta_1=\alpha_1', \beta_2=\alpha_2'$,   the entries of the Fayers matrix $(F)_{ij} = \lf \alpha_i, \alpha'_j \rf$. We have
\[ 
F = \begin{pmatrix}  1 & 0 \\ 
0 & -1  \end{pmatrix} 
\qquad   \hbox{ and }  \qquad \Phi =\left( \begin{pmatrix}  1 & 0 \\ 
0 & -1  \end{pmatrix} F  \begin{pmatrix}  1 & 0 \\ 
0 & -1  \end{pmatrix} \right)^* = \begin{pmatrix}  1 & 0 \\ 
0 & -1  \end{pmatrix}   \ . 
\]
Theorem \ref{thm: detAB}  states that 
 \begin{eqnarray}  \det( A\star B)   & = & A_{\alpha_1} B_{\alpha'_1} - A_{\alpha_2} B_{\alpha_2'} \nonumber \\
 &=  &  \det A_{12} \det A_{34}  \det B_{13} \det B_{24} -  \det A_{13} \det A_{24}  \det B_{12} \det B_{34}     \nonumber \\ 
 &
 = & \begin{vmatrix} a_{11}& a_{12} \\ a_{21} & a_{22} \end{vmatrix}  \begin{vmatrix} a_{31}& a_{32} \\ a_{41} & a_{42} \end{vmatrix} 
  \begin{vmatrix} b_{11}& b_{12} \\ b_{31} & b_{32} \end{vmatrix}  \begin{vmatrix} b_{21}& b_{22} \\ b_{41} & b_{42} \end{vmatrix}    \nonumber\\
  &&   -   
     \begin{vmatrix} a_{11}& a_{12} \\ a_{31} & a_{32} \end{vmatrix}  \begin{vmatrix} a_{21}& a_{22} \\ a_{41} & a_{42} \end{vmatrix}  
    \begin{vmatrix} b_{11}& b_{12} \\ b_{21} & b_{22} \end{vmatrix}  \begin{vmatrix} b_{31}& b_{32} \\ b_{41} & b_{42} \end{vmatrix}  \nonumber
  \end{eqnarray}

\subsubsection{ $m=3$ and $n=2$}We have:
\[ A= \begin{pmatrix}
a_{11}   &   a_{12}   & a_{13}    \\ 
a_{21}   &   a_{22}    & a_{23}  \\
a_{31}  &   a_{32}      &a_{33} \\
a_{41}   &   a_{42} & a_{43} \\
a_{51}   &   a_{52} & a_{53} \\
a_{61}   &   a_{62} & a_{63}     
\end{pmatrix} 
\\  \quad , \quad B= \begin{pmatrix}
b_{11}  &    b_{12}    \\ 
b_{21}  &    b_{22}     \\
b_{31}  &     b_{32}    \\
b_{41}  &     b_{42}    \\
b_{51}  &     b_{52}    \\
b_{61}  &     b_{62}    
\end{pmatrix} 
\]
and hence
\[ A\star B = 
\begin{pmatrix}
a_{11}b_{11}  &   a_{12} b_{11}    & a_{13} b_{11}   & a_{11}b_{12}  &   a_{12} b_{12} &  a_{13} b_{12}  \\ 
a_{21}b_{21}  &   a_{22} b_{21}   & a_{23} b_{21}  & a_{21}b_{22}  &   a_{22} b_{22}   &  a_{23} b_{22}  \\
a_{31}b_{31}  &   a_{32} b_{31}   & a_{33} b_{31}  & a_{31}b_{32}  &   a_{32} b_{32} &   a_{33} b_{32}   \\
a_{41}b_{41}  &   a_{42} b_{41}    & a_{43} b_{41} & a_{41}b_{42}  &   a_{42} b_{42}  & a_{43} b_{42}   \\
a_{51}b_{51}  &   a_{52} b_{51}    & a_{53} b_{51} & a_{51}b_{52}  &   a_{52} b_{52}  & a_{53} b_{52}   \\
a_{61}b_{61}  &   a_{62} b_{61}    & a_{63} b_{61} & a_{61}b_{62}  &   a_{62} b_{62}  & a_{63} b_{62}   \\
\end{pmatrix} 
 \]

The five standard Young tableaux of shape $(2,2,2)$ are:
\[  \underbrace{\begin{Young}
1 & 4 \cr
2  & 5 \cr
3  & 6 \cr
\end{Young}}_{\alpha_1} \quad   , \quad  
\underbrace{\begin{Young}
1 & 3 \cr
2  & 5 \cr
4  & 6 \cr
\end{Young}}_{\alpha_2}
\quad   , \quad  \underbrace{\begin{Young}
1 & 3 \cr
2  & 4 \cr
5  & 6 \cr
\end{Young}}_{\alpha_3}
 \quad   , \quad  
 \underbrace{\begin{Young}
1 & 2 \cr
3  & 5 \cr
4  & 6 \cr
\end{Young}}_{\alpha_4} 
 \quad   , \quad  
\underbrace{\begin{Young}
1 & 2 \cr
3  & 4 \cr
5  & 6 \cr
\end{Young}}_{\alpha_5} 
 \]
Their conjugates $\beta_i = \alpha'_i$ are as follows:
 \[  
\underbrace{\begin{Young}
1 & 2  & 3 \cr
 4 & 5 & 6 \cr
\end{Young}}_{\beta_1}  \quad   , \quad 
 \underbrace{\begin{Young}
1 & 2 & 4 \cr
3 & 5 & 6  \cr
\end{Young}}_{\beta_2}
 \quad   , \quad 
\underbrace{\begin{Young}
1 &2& 5\cr 
3 & 4 & 6 \cr
\end{Young}}_{\beta_3} \quad   , \quad 
 \underbrace{\begin{Young}
1 & 3& 4 \cr
2  & 5 & 6 \cr
\end{Young}}_{\beta_4} 
 \quad   , \quad  \underbrace{\begin{Young}
1 & 3 & 5 \cr
2  & 4 & 6 \cr
\end{Young}}_{\beta_5}
\]
 The Fayers matrix $F_{ij} = \lf \alpha_i,\beta_j\rf$ and the  matrix $\Phi$ are explicitly:
 \[ 
F= \begin{pmatrix} 
      1   &   0 & 0 &0  &0  \\
    0 &  -1   & 0 &0  &0 \\
  0   & 0   &  1 &0 &0 \\
  0  &  0    & 0 & 1 & 0 \\
   -1 & 0    &0  &0  & -1 
\end{pmatrix}
\quad , \quad
\Phi= \begin{pmatrix} 
      1   &   0 & 0 &0  &1 \\
    0 &  -1   & 0 &0  &0 \\
  0   & 0   &  1 &0 &0 \\
  0  &  0    & 0 & 1 & 0 \\
   0 & 0    &0  &0  & -1 
\end{pmatrix}
\]
The bottom left corner entry $F_{5,1}$ comes from the fact that there is\footnote{this tableau is easily found: its entry in row $i$ and column $j$  is   the intersection $\mathrm{col}_i(\beta_1) \cap \mathrm{row}_j(\alpha_5)$ of  the set of numbers in column $i$ of $\beta_1$, with the set of numbers in column $j$ of $\alpha_5$.} a  tableau
\[    
\begin{Young}
1 & 4 \cr
5 & 2 \cr
3  & 6 \cr
\end{Young}
\]
which is not standard, but has  the properties that it is column-equivalent to $\alpha_5$, its conjugate is column-equivalent to $\beta_1$, and has sign $-1$.
Theorem \ref{thm: detAB} states that 
\[ \det(A \star B) = A_{\alpha_1} \left(  B_{\beta_1}   + B_{\beta_5}\right) - A_{\alpha_2} B_{\beta_2} + A_{\alpha_3}B_{\beta_3} + A_{\alpha_4} B_{\beta_4} - A_{\alpha_5} B_{\beta_5} \]
and thus
 \begin{eqnarray}  \det( A\star B)    =      &  + &  \det A_{123}  \det A_{456}  \,  \det B_{14} \det B_{25} \det B_{36}   \nonumber  \\
  &  + &  \det A_{123}  \det A_{456}  \,  \det B_{12} \det B_{34} \det B_{56}   \nonumber  \\
  & -  &  \det A_{124}  \det A_{356}  \,  \det B_{13} \det B_{25} \det B_{46}   \nonumber  \\
   &  + &  \det A_{125}  \det A_{346}  \,  \det B_{13} \det B_{24} \det B_{56}   \nonumber  \\
 & +  & \det A_{134} \det A_{256}  \,  \det B_{12} \det B_{35} \det B_{46}   \nonumber  \\
      &     - & \det A_{135} \det A_{246}  \,  \det B_{12} \det B_{34} \det B_{56} \nonumber 
  \end{eqnarray}

%
%
%
%

\subsection{Proof of theorem \ref{thm: detAB}}
Let $V, W$ be vector spaces  over $\Q$  of dimensions $m$ and $n$ respectively.  Consider a set of $mn$ vectors $v_1,\ldots, v_{mn} \in V$ and  $w_1, \ldots, w_{mn} \in W$. If we fix a basis $e_1,\ldots, e_m$ of $V$ and $f_1,\ldots, f_n$ of $W$, the vectors $v_i$ and $w_j$   for $1\leq i,j\leq mn$ may be written 
\[ v_i = a_{i1} e_1 + \ldots + a_{im} e_m\] 
corresponding to the $i^\mathrm{th}$ row of an $mn \times m$ matrix $A$, 
and similarly $w_i =   b_{i1} f_1 + \ldots + b_{in} f_n$ corresponding to the $i^\mathrm{th}$ row of an $mn \times n$ matrix $B$. 

The tensor product $V\otimes W$ has dimension $mn$ and basis $e_i \otimes f_j$ for $1\leq i\leq m$ and $1\leq j \leq n$. The vector $v_i \otimes w_i$ has coordinates
\[ v_i \otimes w_i = a_{i1} b_{i1} \,  e_1\otimes f_1 + \ldots + a_{im} b_{in} \, e_m \otimes f_n \]
which corresponds to the $i^\mathrm{th}$ row of the matrix $A\star B$, see \eqref{KroneckerTensorConvention}.   
Define an element
\[   
\varpi(v,w)   \in    \bwedge^{mn} \left( V\otimes W\right) \]
by the formula
\begin{equation}  \label{inproofomegavw}  
\varpi(v,w) =   \left( v_1\otimes w_1\right) \wedge  \left( v_2\otimes w_2\right) \wedge \ldots \wedge \left( v_{mn}\otimes w_{mn}\right)\ .  
\end{equation} 
Then we have
 \begin{equation} \label{piece1}  \varpi(v,w)   =  \det(A\star B) \,  (e_1 \otimes f_1) \wedge (e_2 \otimes f_1) \wedge \ldots  \wedge (e_m \otimes f_n) \end{equation}
 and therefore,  computing the determinant $\det(A\star B)$ is equivalent to finding an expression for $ \varpi(v,w)$. 
Recall that Schur-Weyl duality states that 
\[ V^{\otimes p} = \bigoplus_{\lambda\vdash p}  \mathbb{S}_{\lambda} V \otimes S^{\lambda}\]
 where the direct sum is over all partitions $\lambda$ of $p$,  $\mathbb{S}_{\lambda} $ denotes the Schur functor associated to $\lambda$, and $S^{\lambda}$ the corresponding Specht module, which  is a representation of the symmetric group $\Sigma_p$ acting upon $V^{\otimes p}$ by permuting the factors. 
One has
 \[ V^{\otimes p} \otimes W^{\otimes p}  = \bigoplus_{\lambda, \mu \vdash p}  \mathbb{S}_{\lambda} V \otimes  \mathbb{S}_{\mu} W  \otimes S^{\lambda} \otimes S^{\mu} \ . \]
The exterior algebra $\bwedge^{p} ( V\otimes W)$ is the subspace of $V^{\otimes p} \otimes W^{\otimes p}$  which corresponds  to the  alternating (sign) representation $\varepsilon_p$ of the copy of the symmetric group $\Sigma_p$ which is embedded diagonally $\Sigma_{p} \hookrightarrow \Sigma_p \times \Sigma_p$, where $\Sigma_p \times \Sigma_p$  acts naturally on $V^{\otimes p} \otimes W^{\otimes p}$. 
 Since $S^{\lambda} \cong (S^{\lambda'} \otimes \varepsilon)^{\vee}$ (see below) it follows from Schur's lemma that the sign representation $\varepsilon$ occurs in 
  the tensor product of Specht modules $S^{\lambda} \otimes S^{\mu}$ if and only if $\mu = \lambda'$, with multiplicity  one,  and hence one has the formula 
   \[  \bwedge^p \left( V \otimes W \right)  = \bigoplus_{\lambda\vdash p}  \mathbb{S}_{\lambda} V \otimes  \mathbb{S}_{\lambda'} W    \ , \]
 where both sides of the equation are naturally embedded  in $V^{\otimes p}\otimes W^{\otimes p}$. In  the case of interest, we  set  $p=mn$ and note that 
 $\mathbb{S}_{\lambda} V$ vanishes if $\lambda$ has more than $m=\dim V$ rows. Similarly  
  $\mathbb{S}_{\lambda'} W$ vanishes if $\lambda$ has more than $n=\dim W$ columns. Since $\lambda \vdash mn$ it follows that the only contributing terms to the right-hand side of the previous formula are those where $\lambda = (n,\ldots, n)$ is a rectangular array with $m$ rows and $n$ columns and hence the formula reduces to a single term
  \[  \bwedge^{mn} \left( V \otimes W \right)  =    \mathbb{S}_{\lambda} V \otimes  \mathbb{S}_{\lambda'} W \quad \hbox{ where } \lambda = n^m\ . \]
   There is a  canonical isomorphism 
   $  \mathbb{S}_{\lambda} V \cong \mathrm{Sym}^n 
  \!\left(\bwedge^m V\right) $, and 
  $\mathbb{S}_{\lambda'} W  \cong  \mathrm{Sym}^m ( \bwedge^n W)$, since both sides of each equation are one-dimensional vector spaces. Consequently, we deduce that there is a canonical  isomorphism of one dimensional  $\Q$-vector spaces:
    \begin{equation} \label{inproof: naturalmap}   \bwedge^{mn} \left( V \otimes W \right)  \overset{\sim}{\To}    \mathbb{S}_{\lambda} V \otimes  \mathbb{S}_{\lambda'} W=\Sym^n \! \left( \bwedge^m V\right)  \otimes  \Sym^m \!\left( \bwedge^n W\right)
   \end{equation} 
which  sends
   \[ (e_1 \otimes f_1) \wedge (e_1 \otimes f_2) \wedge \ldots  \wedge (e_m \otimes f_n)  \mapsto  (e_1 \wedge \ldots \wedge e_m)^n \otimes (f_1 \wedge \ldots \wedge f_n)^m \ . \] 
  We wish to compute the image of \eqref{inproofomegavw} under the map \eqref{inproof: naturalmap}, 
    which  corresponds to the inclusion of the sign representation $\varepsilon_p$ in the tensor product of Specht modules $S^{\lambda} \otimes S^{\lambda'}$.     Recall that $S^{\lambda}$ has a basis indexed by the set of standard Young tableaux $\alpha$ of shape $\lambda$. 
 A standard tableau  $\alpha$ of shape $\lambda = n^m$  corresponds to a linear map:
   \begin{eqnarray} \label{inproof:alphadet} 
   \alpha:  V^{\otimes mn} & \To & \mathbb{S}_{\lambda} V \cong \mathrm{Sym}^n 
  \,  (\bwedge^m V)   \\ 
 v_1 \otimes v_2 \otimes \ldots \otimes v_{mn}  &   \mapsto  &  \prod_{i=1}^n  \bigwedge_{i\in \mathrm{col}_i(\alpha)} v_i   \nonumber 
   \end{eqnarray} 
     where the exterior products are over the set of $m$ elements in each column of $\alpha$ (in increasing order) and the (symmetric) product  is over all $n$ rows. A similar statement holds for  a standard tableau $\beta$ of shape $m^n$. 
 
 It follows from \cite[Proposition 4.8]{Fayers}, that the following  bilinear map, which is defined on standard tableaux $
 \alpha, \beta$ of respective shapes $\lambda, \lambda'$,
 \begin{eqnarray} \label{ProjectStensSToEpsilon}
S^{\lambda} \otimes S^{\lambda'}  & \To & \varepsilon   \\ 
\alpha \otimes \beta & \mapsto & F_{\alpha, \beta} = \lf \alpha, \beta \rf \nonumber 
 \end{eqnarray} 
 is a surjective morphism of Specht modules.  Equivalently, by tensor-Hom adjunction, the matrix $F$ defines an isomorphism of Specht modules:
 \[ F: S^{\lambda'} \overset{\sim}{\To}   (S^{\lambda})^{\vee} \otimes \varepsilon\ .\]
Its inverse transpose  defines an isomorphism
 \[ F^*: \left(S^{\lambda'}\right)^{\vee} \overset{\sim}{\To}  S^{\lambda} \otimes \varepsilon \ . \]
 Again by adjunction, we deduce that the 
 morphism of Specht modules:
 \begin{eqnarray}
\left(S^{\lambda} \right)^{\vee} \otimes 
\left(S^{\lambda'} \right)^{\vee} & \To & \varepsilon\nonumber \\ 
\alpha^{\vee} \otimes \beta^{\vee} & \mapsto & F^*_{\alpha^{\vee}, \beta^{\vee}} \nonumber 
 \end{eqnarray}   
 is computed by  $F^*$.  Consider the composition of maps
 \begin{equation} \label{inproof:counitmaps} 1 \To S^{\lambda} \otimes (S^{\lambda})^{\vee} \otimes S^{\lambda'} \otimes (S^{\lambda'})^{\vee} \To (S^{\lambda}\otimes S^{\lambda'} )\otimes \varepsilon
 \end{equation} 
 where the first map is given by a  tensor product of counits, and the second map contracts the second and fourth terms in the tensor product using  $F^*$.
 It is a morphism of Specht modules, where $1$ denotes the trivial representation.   Recall that the counit $\Q \rightarrow V\otimes V^{\vee}$ for a  finite dimensional vector space $V$ may be represented by 
 $1 \mapsto \sum_{i} v_i \otimes v_i^{\vee}$, where $v_i$ is any basis of $V$, and $v_i^{\vee}$ the dual basis. 
 Applying this to the basis $\{\alpha\}$ for $S^{\lambda}$ and $\{\beta\}$ for $S^{\lambda'}$ we may compute   \eqref{inproof:counitmaps} by 
 \[ 1 \mapsto  \sum_{\alpha,\beta} \alpha \otimes \alpha^{\vee} \otimes \beta \otimes \beta^{\vee} \mapsto  \sum_{\alpha, \beta} F^*_{\alpha^{\vee}, \beta^{\vee}} \, \alpha\otimes \beta \ .   \]
 Finally,  it follows that the map of Specht modules obtained 
  by tensoring  \eqref{inproof:counitmaps} with $\varepsilon$ is given explicitly by the map
 \begin{eqnarray}  \label{varepsilonSection} 
 \varepsilon &  \To   &  S^{\lambda}\otimes S^{\lambda'}  \\
 1 & \mapsto & \sum_{\alpha, \beta} \Phi_{\alpha^{\vee}, \beta^{\vee}}\,  \alpha\otimes \beta  \ ,   \nonumber 
 \end{eqnarray}  
since tensoring with $\varepsilon$ corresponds to multiplying $F^*$ by $\mathcal{E}$ on both the left and the right.   We deduce that the map \eqref{inproof: naturalmap}, which may be rewritten
\[     \bwedge^{mn} \left( V \otimes W \right)  \otimes \varepsilon  \overset{\sim}{\To}    \mathbb{S}_{\lambda} V \otimes S^{\lambda}  \otimes  \mathbb{S}_{\lambda'} W  \otimes S^{\lambda'} \]
 is induced by the natural map
 of Specht modules $ \varepsilon\rightarrow   S^{\lambda}\otimes S^{\lambda'}$, 
 and sends 
 \begin{eqnarray}   \label{piece2} \varpi(v,w) & \mapsto  &  \sum_{\alpha, \beta}  \Phi_{\alpha^{\vee}, \beta^{\vee}} \, (\alpha \otimes \beta )(\varpi(v,w))  \\
  & = &  \sum_{\alpha, \beta}  \Phi_{\alpha^{\vee}, \beta^{\vee}}\,   \prod_{i=1}^n \left( \bigwedge_{i\in \mathrm{col}_i(\alpha)} v_i \right) \,  \prod_{i=1}^m  \left( \bigwedge_{i\in \mathrm{col}_i(\beta)} w_i \right) \nonumber 
 \end{eqnarray} 
  where the second line follows from  \eqref{inproof:alphadet}.
 By definition of the determinant of a matrix, one has 
 \begin{equation} \label{piece3} \prod_{i=1}^n   \bigwedge_{i\in \mathrm{col}_i(\alpha)} v_i  = \left( \prod_{i=1}^n  \det A_{\mathrm{col}_i(\alpha)}  \right)   e_1\wedge e_2 \wedge \ldots \wedge e_n
 \end{equation}
 and a similar statement for $B$. Putting the pieces \eqref{piece1},\eqref{piece2},\eqref{piece3} together  proves the theorem. 
 
  \subsection{Proof of theorem \ref{thm: detABSymmetric}}
 Follow the argument of the previous paragraph but instead replace the morphism  \eqref{varepsilonSection}
with  a morphism 
  \[ \varepsilon \To S^{\lambda} \otimes S^{\lambda'} \]
constructed as follows. For  standard tableaux $\alpha, \beta$  of shape $\lambda, \lambda'$ respectively,  define
\begin{eqnarray} 
\rho_{\alpha,\beta}  & \in &  S^{\lambda} \otimes S^{\lambda'} \nonumber \\
\rho_{\alpha,\beta} & = & \sum_{\sigma \in \Sigma_{mn}}  \varepsilon(\sigma) \sigma(\alpha) \otimes \sigma(\beta)\nonumber
\end{eqnarray} 
 The image of $\rho_{\alpha, \beta} $  under the morphism \eqref{ProjectStensSToEpsilon} is equal to $(mn)! \lf \alpha,\beta \rf$, since one has $\lf \sigma\alpha, \sigma \beta\rf = \varepsilon(\sigma) \lf \alpha, \beta \rf$. By Schur's lemma $\rho_{\alpha,\beta}$ is non-zero if and only if $\lf \alpha, \beta \rf$ is non-zero. 
 Since the standard tableaux form a $\Z$-basis for $S^{\lambda}$, it follows from \eqref{inproof:counitmaps} that $\rho_{\alpha,\beta}$ is defined over the integers, and hence there exists an integer $\kappa_{\alpha,\beta}$ such that
 \[ \rho_{\alpha,\beta} = \kappa_{\alpha, \beta} \left( \sum_{\gamma, \delta} \Phi_{\gamma^{\vee}, \delta^{\vee}} \gamma \otimes \delta \right)\ .\]

  For the last part of the theorem,  set $\alpha =1_{\lambda}$ and $\beta = 1'_{\lambda}$ and let $A, B$ be the matrices \eqref{KroneckerCase}. Then $A_{\sigma \alpha}$ is non-zero if and only if $t=\sigma \alpha$  is a tableau which is row equivalent to $1'_{\lambda}$, in which case it equals $ \pm \det(C)^n$ with sign  given by the coefficient of the tabloid $\{t\} $ in $e_{1'_{\lambda}}$ (in the notation of  \cite{Fayers} and references therein).
   Similarly, $B_{\sigma \beta}$ is non-zero if and only if $t'$ is  row-equivalent to $1_{\lambda}$, and equal to $\pm \det(D)^m$ with sign given by the coefficient of $\{t'\}$ in $e_{1_{\lambda}}$.  Since $\lf t,t' \rf = \lf \sigma \alpha, \sigma \beta \rf = \varepsilon(\sigma) \lf \alpha,\beta \rf = \varepsilon(\sigma)$, we deduce that 
   \[ \sum_{\sigma} \varepsilon(\sigma) A_{\sigma \alpha} B_{\sigma \beta} =  \lf  e_{1_{\lambda}}, e_{1'_{\lambda}}  \rf \, \det(C)^n \det(D)^m \ .\]
 By \cite[Lemma 4.6]{Fayers}, the quantity $ \lf   e_{1_{\lambda}}, e_{1'_{\lambda}}  \rf $  is equal to  the product of  the Hook lengths of the diagram $\lambda$. The statement follows since
 $\det(A\star B) = \det(C \otimes D) = \det(C)^n \det(D)^m$. 

  \section{Proofs of theorems \ref{thm: FirstVDMformula} and  \ref{thm: SecondFormula}}
Recall the notation from section \S\ref{sect: firstformula}. Given $N_1,\ldots, N_r$ integers $\geq 2$, and any number $\z_1,\ldots, \z_K$ of vectors of length $r$, denote by 
\[ V_{(N_1,\ldots, N_r)} (\z_1,\ldots, \z_K)\]
the matrix with $N_1N_2\ldots N_r$ columns and $K$ rows whose entries in the $i^{\mathrm{th}}$ row are the monomials in the coordinates of $\z_i$, ordered degree-lexicographically.

\begin{lem} The Vandermonde matrix is an amalgamated  product:
\[V_{(N_1,\ldots, N_r)} (\z_1,\ldots, \z_{mn}) = V_{(N_1,\ldots, N_k)}(\z'_1,\ldots, \z'_{mn}) \star V_{(N_{k+1},\ldots, N_r)}(\z''_1,\ldots, \z''_{mn}) \]
\end{lem} 

\begin{proof} This holds by definition of $\star$ since 
the degree-lexicographically ordered row vector of monomials $(e_{1},\ldots, e_{mn})(\z)$ in $\z=(\z',\z'')$ is equal to the Kronecker tensor product of the row vectors of monomials
$(e_1,\ldots, e_m)(\z')$ and $(e_1,\ldots, e_n)(\z'')$.  
\end{proof} 
 
Theorems 
 \ref{thm: FirstVDMformula} and  \ref{thm: SecondFormula} therefore follow  from theorems \ref{thm: detAB} and \ref{thm: detABSymmetric} respectively.

\begin{example} \label{example:V32}
We have
\[ V_{3,2}(x,y) =  \begin{pmatrix} 1 & x_1 & x_1^2 &  y_1 & x_1y_1 & x^2_1y_1 \nonumber\\ 
1 & x_2 & x_2^2 &y_2 & x_2y_2& x^2_2y_2 \nonumber\\ 
1 & x_3 & x_3^2& y_3 & x_3y_3& x^2_3y_3 \nonumber\\ 
1 & x_4 & x_4^2 & y_4 & x_4y_4& x^2_4y_4 \nonumber\\ 
1 & x_5 & x_5^2 & y_5 & x_5y_5& x^2_5y_5 \nonumber\\ 
1 & x_6 & x_6^2 & y_6 & x_6y_6& x^2_6y_6 \nonumber\\ 
\end{pmatrix} = \begin{pmatrix} 1 & x_1 & x_1^2  \\
 1& x_2& x_2^2 \\
  1& x_3 & x_3^2 \\ 
  1& x_4 & x_4^2  \\ 
   1& x_5 & x_5^2  \\ 
    1& x_6 & x_6^2  
  \end{pmatrix} \star \begin{pmatrix} 1 & y_1 \\ 1& y_2 \\ 1& y_3 \\ 1& y_4 \\ 1& y_5 \\ 1& y_6 \end{pmatrix}   \]
Consequently, theorem \ref{thm: detAB} implies that 
\[ \det(V_{3,2}(x,y))  = \qquad \qquad \qquad \]
\[ + \Big((x_3-x_1)(x_3-x_2)(x_2-x_1)(x_6-x_4)(x_6-x_5)(x_5-x_4) \Big)\Big((y_4-y_1)(y_5-y_2)(y_6-y_3)\Big)\]
\[ + \Big((x_3-x_1)(x_3-x_2)(x_2-x_1)(x_6-x_4)(x_6-x_5)(x_5-x_4) \Big)\Big((y_2-y_1)(y_4-y_3)(y_6-y_5)\Big)\]
\[ - \Big((x_4-x_1)(x_4-x_2)(x_2-x_1)(x_6-x_3)(x_6-x_5)(x_5-x_3) \Big)\Big((y_3-y_1)(y_5-y_2)(y_6-y_4)\Big)\]
\[ + \Big((x_5-x_1)(x_5-x_2)(x_2-x_1)(x_6-x_3)(x_6-x_4)(x_4-x_3) \Big)\Big((y_3-y_1)(y_4-y_2)(y_6-y_5)\Big)\]
\[ + \Big((x_4-x_1)(x_4-x_3)(x_3-x_1)(x_6-x_2)(x_6-x_5)(x_5-x_2) \Big)\Big((y_2-y_1)(y_5-y_3)(y_6-y_4)\Big)\]
\[ - \Big((x_5-x_1)(x_5-x_3)(x_3-x_1)(x_6-x_2)(x_6-x_4)(x_4-x_2) \Big)\Big((y_2-y_1)(y_4-y_3)(y_6-y_5)\Big)\]
This formula has the consequence that $V_{3,2}(x,y)$ factorises under  mild conditions on $x_i,y_i$. For instance, if $y_1=y_2$ and $x_6=x_3$ then all terms vanish except the first and $V_{3,2}(x,y)$ completely factorises as a product of linear factors.

Thus we may deduce some simple sufficient geometric conditions under which a configuration of six points in the plane $p_i=(x_i,y_i)$ for $1\leq i \leq 6$ may be interpolated by a polynomial of the form $f(x,y) = c_0 + c_1 x + c_2 y + c_3 x^2 + c_4 xy + c_5 x^2 y$ since this is equivalent to the vanishing of $\det\, V_{3,2}(x,y)$.  For example, if  $p_1,p_2$ have the same $y$ coordinate and $p_3,p_4, p_6$ the same $x$-coordinate. One may generate an unlimited number of increasingly baroque examples of this kind. 
\end{example}

\section{Homogeneous Vandermonde determinants}
Theorem \ref{thm: homogeneous} follows from theorems \ref{thm: detAB} or \ref{thm: detABSymmetric} by noting that the homogenous Vandermonde matrix  \eqref{HomVandermonde} is a submatrix of the rectangular version \eqref{defn: VNmatrix}. Its determinant can therefore be deduced from that of $\det (A \star B)$ by taking coefficients (or partial derivatives) with respect to a certain set of $a_{ij}, b_{ij}$ and finally specialising $a_{ij}, b_{ij}$ as in the previous section. Rather than writing down a general formula we illustrate with two different ways of computing \eqref{Vhomconic}.

First note that the determinant in question may be computed indirectly  by  first considering the amalgam   of the following two $9 \times 3$ matrices:
\[ A =  \begin{pmatrix} a_{1,1} & a_{1,2} & a_{1,3} \\ 
a_{2,1} & a_{2,2} & a_{2,3} \\ 
     \vdots & \vdots & \vdots \\ 
      a_{9,1} & a_{9,2} & a_{9,3} \\ 
        \end{pmatrix} \qquad    B= \begin{pmatrix} b_{1,1} & b_{1,2} & b_{1,3} \\ 
        b_{2,1} & b_{2,2} & b_{2,3} \\ 
     \vdots & \vdots & \vdots \\ 
      b_{9,1} & b_{9,2} & b_{9,3} \\ 
        \end{pmatrix}    \]
which for later reference is 
\[    \begin{pmatrix} a_{1,1} b_{1,1} & a_{1,2} b_{1,1} & a_{1,3} b_{1,1} & 
a_{1,1} b_{1,2} & a_{1,2} b_{1,2} & a_{1,3} b_{1,2} &
a_{1,1} b_{1,3} & a_{1,2} b_{1,3} & a_{1,3} b_{1,3} \\
a_{2,1} b_{2,1} & a_{2,2} b_{2,1} & a_{2,3} b_{2,1} & 
a_{2,1} b_{2,2} & a_{2,2} b_{2,2} & a_{2,3} b_{2,2} &
a_{2,1} b_{2,3} & a_{2,2} b_{2,3} & a_{2,3} b_{2,3} \\
     \vdots &  \vdots&\vdots&\vdots&\vdots&\vdots&\vdots&\vdots&\vdots \\ 
  a_{9,1} b_{9,1} & a_{9,2} b_{9,1} & a_{9,3} b_{9,1} & 
a_{9,1} b_{9,2} & a_{9,2} b_{9,2} & a_{9,3} b_{9,2} &
a_{9,1} b_{9,3} & a_{9,2} b_{9,3} & a_{9,3} b_{9,3} \\
        \end{pmatrix}\]
        
\begin{example} Consider the amalgam $A\star B$ and take the coefficient of the variables
$a_{9,3} ,a_{8,3}, a_{7,2}, b_{7,3} , b_{8,3} , b_{9,2}$, or equivalently, 
compute
\[  \frac{\partial}{\partial a_{9,3}} \frac{\partial}{\partial a_{8,3}}  \frac{\partial}{\partial a_{7,2}}  \frac{\partial}{\partial b_{7,3}}  \frac{\partial}{\partial b_{8,3}}  \frac{\partial}{\partial b_{9,2}}  \det(A\star B) \] 
using theorem \ref{thm: detAB}. 
This amounts to computing the  $6\times 6$ minor of $A \star B$ obtained by deleting rows $7,8,9$ and  columns $6,8,9$ which contain entries  $a_{i,2} b_{i,3}$, $a_{i,3} b_{i,2}$ and $a_{i,3} b_{i,3}$.  Now set $a_{i,1}=1, a_{i,2}=x_i, a_{i,3}=x_i^2$, and 
$b_{i,1}=1, b_{i,2}=y_i, b_{i,3}=y_i^2$. This gives an expression for 
$ \det V_{2,2}^{\mathrm{hom}} (p_1,\ldots, p_6) $
as a sum of 24 factorising terms:
\begin{eqnarray}   V_{125}(x)V_{123}(y) (x_3-x_4) (y_4-y_6) & +  &   V_{125}(x)V_{134}(y)(x_3-x_6)(y_2-y_6)   \nonumber \\
 + V_{126}(x)V_{135}(y)     (x_3-x_4)(y_1-y_2) & +  &   V_{123}(x)V_{124}(y)(x_4-x_5)(y_3-y_6)   \nonumber \\
+ V_{136}(x)V_{124}(y)     (x_2-x_5)(y_3-y_5) & +  &   V_{145}(x)V_{123}(y)(x_2-x_6)(y_4-y_6)   \nonumber \\
+   \hbox{  18 similar terms, } \nonumber 
\end{eqnarray} 
where  \[ V_{ijk}(z) = \det \begin{pmatrix} 1 & z_i & z_i^2 \\ 1 & z_j& z_j^2 \\ 1 & z_k & z_k^2 \end{pmatrix} = ( z_j-z_i )(z_k-z_j)(z_k-z_i) \ .\]
On specialising $x_2=x_1$ the formula collapses down to the eight terms listed in example \ref{Ex: Pascalseparate}.
\end{example}

\begin{example} \label{example: Pascalconic} Our determinant formula for an amalgam can be used in a different manner. We illustrate again with the  simple example
\ref{example: Pascal}.  Consider the amalgam $A \star B$ as above, but this time  take the coefficient of  $a_{9,2}, a_{8,1}, a_{7,1}, b_{9,3}, b_{8,3}, b_{7,2}$ in its determinant, or equivalently,  compute
\[  \frac{\partial}{\partial a_{7,1}} \frac{\partial}{\partial a_{8,1}}  \frac{\partial}{\partial a_{9,2}}  \frac{\partial}{\partial b_{7,2}}  \frac{\partial}{\partial b_{8,3}}  \frac{\partial}{\partial b_{9,3}}  \det(A\star B) \] 
which amounts to computing  a certain $6\times 6$ minor of $A \star B$, and 
finally substituting $a_{i,1}=b_{i,1}=1$, $a_{i,2}=b_{i,2}=x_i$, $a_{i,3}=b_{i,3}=y_i$ for $i=1,\ldots, 6$.\footnote{If one performs the substitutions before taking coefficients, the amalgam $A \star B$ has first row 
\[  \begin{pmatrix} 1 & x_1  & y_1 \end{pmatrix}  \otimes  \begin{pmatrix} 1 & x_1  & y_1 \end{pmatrix} =  \begin{pmatrix} 1 & x_1  & y_1 & x_1 & x_1^2 & x_1 y_1 & y_1 & x_1y_1 & y_1^2  \end{pmatrix} \]
whose determinant vanishes because there are 3 repeated columns. The $6 \times 6$ minor of $A \star B$ described above is obtained by striking out the 3 columns $4,7,8$  with repeated entries $x_1, y_1$ and $x_1y_1$, together with the three final rows $7,8,9$. } 
This is by no means the only way of computing the determinant as a minor of an amalgam.  By applying theorem \ref{thm: FirstVDMformula} to compute $\det (A \star B)$ and 
 carrying out this procedure,  we
find  a curious expression for  the determinant of $V(p_1,\ldots, p_6)$  as a sum of  completely factorising terms. It becomes particularly simple if we assume $x_6=y_6=0$, i.e., the point $p_6$ is at the origin and can easily be written down. 

To do so compactly,    let us define 
\[ V_{i,j,k} = \det \begin{pmatrix} 1 & x_i & y_i \\  1 & x_j & y_j \\ 1& x_k & y_k \end{pmatrix}  \ . \] 
This determinant vanishes if and only if $p_i,p_j,p_k$ are colinear. Let us also write 
\[  c_{i,j} =V_{i,j,6}=    \det \begin{pmatrix}  x_i & y_i \\  x_j & y_j  \end{pmatrix} =  x_i y_j - x_j y_i \ . \]
Then we find that the determinant of $V(p_1,\ldots, p_5,0)$ simplifies to  a sum of six  completely factorising terms:  
 \begin{eqnarray}  V(p_1,\ldots, p_5, 0 ) &=&  -  
\det \begin{pmatrix}   x_1 & y_1 & x_1^2 & x_1y_1 & y_1^2  \\ 
 x_2 & y_2 & x_2^2 & x_2y_2 & y_2^2  \\
 x_3 & y_3 & x_3^2 & x_3y_3 & y_3^2  \\
 x_4 & y_4 & x_4^2 & x_4y_4 & y_4^2  \\
 x_5 & y_5 & x_5^2 & x_5y_5 & y_5^2  
\end{pmatrix}  \nonumber \\
 & = & \qquad  V_{1,2,5}  \, c_{1,3} c_{2,4} \, y_5 (x_4-x_3)   +V_{1,2,3} \,    c_{1,4} c_{2,5}  y_3 (x_5-x_4)       \nonumber   \\ 
   &  & +  \quad \,   V_{1,2,3} \, c_{1,2} c_{3,4} y_5 (x_5-x_4)      +   V_{1,3,4}  \, c_{1,2} c_{3,5}  y_4 (x_5-x_2)  \nonumber \\ 
    &  &  -  \quad  \, V_{1,3,5} \, c_{1,2} c_{3,4}   y_5(x_4-x_2)  -   V_{1,2,4}  \, c_{1,3} c_{2,5} y_4 (x_5-x_3) \ .   \nonumber
\end{eqnarray}

\end{example}

This  general technique may be applied to  matrices of Vandermonde type which are of not of the precise form \eqref{defn: VNmatrix}, and give sufficient conditions for a set of points in projective space to be interpolated by  homogeneous polynomials of given degree.

\section{An application to transfinite diameters} \label{sect: applicationtransfinite}

Let $K\subset \C^r$ be a compact subset. One of the many  possible variants of the  notion of  transfinite diameter of $K$ is 
\[ t_r(K) = \lim_{N\rightarrow \infty} \sup_{\z_1,\ldots, \z_{N^r} \in K}  \,  \left| \det V_{(N,\ldots, N)}(\z_1,\ldots, \z_{N^r})\right|^{1/D_N} \]
where 
$D_N =  \frac{r}{2} N^r(N-1)$ is the total degree of the determinant.
The existence of the limit may be established following  \cite{Zaharjuta}.  Note that a more common convention is to use the homogeneous version of the Vandermonde matrix $V_{N,r}^{\mathrm{hom}}$ \eqref{HomVandermonde}.

Consider the following variant. Let $\underline{w} =(w_1,\ldots, w_r)$ denote a vector of positive rational numbers. If the limit exists, we may define
\[  t_{\underline{w}} (K) = \lim_{N\rightarrow \infty} \sup_{\z_1,\ldots, \z_{\ell} \in K}  \,  \left|  \det V_{(w_1 N,\ldots, w_r N)}(\z_1,\ldots, \z_{\ell})\right|^{1/D_{\underline{w} N}}  \]
where $\ell= w_1\ldots w_r N^r$ and $D_{\underline{w}N}$ is the total degree of the determinant\footnote{We take the limit over all $N$ which are divisible by the common denominator of $w_1,\ldots, w_r$ so that all $w_iN$ are integers}. When $\underline{w}=(1,\ldots, 1)$  this reduces to the previous definition, since    $t_{(1,\ldots, 1)}(K) = t_r(K)$.

The following theorem was established a long time ago \cite{SchifferSiciak}  in a mildly different context (see also \cite{BloomCalvi}) using different methods, but is an immediate first application of the formula \eqref{Intro:SecondVDMformula} for the multivariable Vandermonde determinant. 
\begin{thm}
Let $K\subset \C^r$ be equal to a product $K= K_1\times K_2$ of compact subsets of $\C^{k} \times \C^{r-k}$, where $k\geq 0$. 
Write $\underline{w} = ( \underline{w}', \underline{w}'')$ where 
$\underline{w}' = (w_1,\ldots, w_k)$ and  $ \underline{w}'' = (w_{k+1},\ldots, w_r).$ If the limits exist we have
\[   t_{\underline{w}}(K) = t_{\underline{w}'}(K_1)^{|\underline{w}'|/|\underline{w}|} \times t_{\underline{w}''}(K_2)^{|\underline{w}''|/|\underline{w}|}  \]
 where $|\underline{w}| = w_1+ \ldots+ w_r$ and similiarly for $|\underline{w}'|$, $|\underline{w}''|$. 
In particular, 
\[ t_r(K) = t_k(K_1)^{k/r} \times t_{r-k}(K_2)^{(r-k)/r}  \ . \]
\end{thm}
\begin{proof} We may assume that $N_i=w_i N $ is an integer for all $i$.  Let $m=  N_1\ldots N_k$ and $n=  N_{k+1}\ldots N_r$. Let us denote by 
\begin{equation}
D_{(N_1,\ldots ,N_r)}  =      \deg (\det V_{(N_1,\ldots,  N_r)}(\z_1,\ldots, 
\z_{mn}) \ . 
\end{equation}
Then we have 
\begin{eqnarray}  D_{(N_1,\ldots ,N_r)}  &=  & \sum_{ 0 \leq i_1\leq  N_1-1, \ldots,   0 \leq i_r\leq N_r-1} i_1+ \ldots +i_r   \\
& =&  \frac{N_1 N_2 \ldots N_r}{2} ( N_1+ N_2 + \ldots +  N_r - r )  \nonumber
\end{eqnarray} 
from which it follows that 
\[  D_{(N_1,\ldots ,N_r) }  = n  D_{(N_1,\ldots ,N_k) }   +   m D_{(N_{k+1},\ldots ,N_r)}  \]
which is consistent with theorems \ref{thm: FirstVDMformula} and \ref{thm: SecondFormula}.

 By theorem  \ref{thm: SecondFormula}, we have
\[  \left| \sup_{\z_i \in K} V_{(N_1,\ldots, N_r)} (\z_i) \right| \leq \frac{(mn)!}{H_{m,n}} \left|  \sup_{\z'_i \in K_1} V_{(N_1,\ldots,N_k)}(z'_i) \right|^{n} \left|  \sup_{\z''_j \in K_2} V_{( N_{k+1},\ldots, N_r)} (\z''_j)\right|^{m}   \ .\]
Let us write 
\[ \alpha  = \frac{ N_1 + \ldots +  N_k - k } {N_1+ \ldots +N_r- r} \quad \hbox{ and } \quad \beta= \frac{  N_{k+1} + \ldots + N_r -(r-k)} {N_1+ \ldots +N_r-r}  \]
where $\alpha+ \beta =1$. Then
\begin{multline*}   \left| \sup_{\z_i \in K} V_{( N_1,\ldots, N_r)} (\z_i) \right|^{1/D_{(N_1,\ldots, N_r)}} \leq  \left( \frac{(mn)!}{H_{m,n}}\right)^{1/D_{(N_1,\ldots, N_r)}}\\
 \times  \left|  \sup_{\z'_i \in K_1} V_{(N_1,\ldots, N_k)}(\z'_i) \right|^{\alpha / D_{(N_1,\ldots, N_k) }} \left|  \sup_{\z''_j \in K_2} V_{(N_{k+1},\ldots,  N_r)} (\z''_j)\right|^{\beta / D_{(N_{k+1},\ldots, N_{r})}} \end{multline*}
 It  follows from  Stirling's formula that
  \[   \left( \frac{(mn)!}{H_{m,n}}\right)^{1/D_{(N_1,\ldots, N_r) }}  = \left( \left(\frac{(mn)!}{H_{m,n}}\right)^{1/mn} \right)^{2/{(N_1+\ldots+ N_r-r)}} \To 1\]
 as $N\rightarrow \infty$, since    $N_i = w_i N$.  Therefore, by letting $N\rightarrow \infty$ we deduce that 
  \[  t_{\underline{w}}(K) \leq  t_{\underline{w}'}(K_1)^{|\underline{w}'|/|\underline{w}|} \times t_{\underline{w}''}(K_2)^{|\underline{w}''|/|\underline{w}|} \]
since $\alpha \rightarrow   (w_1+\ldots +w_k)/(w_1+\ldots +w_r)  $ and $\beta \rightarrow (w_{k+1}+\ldots +w_r)/(w_1+\ldots +w_r)$.

To prove a lower bound,  note that since $K_1, K_2$ are compact, there exists a configuration of points $\z'_1,\ldots, \z'_{m} \in K_1$ and $\z''_1,\ldots, \z''_n \in K_2$ such that 
the respective multivariable Vandermonde determinants are maximized.  
Define a configuration of $mn$ points in $K$ as follows. For every  $1\leq k \leq mn$ write
\[ k= i+ (j-1)  m  \]
for unique  $1\leq  i \leq m$ and $1\leq j\leq n$. Set 
\[ \z_k = (\z'_i, \z''_j)\]
and observe that with these choices, 
\[ V_{(N_1,\ldots, N_r)}) (\z_1,\ldots, \z_{mn}) =  V_{(N_1,\ldots, N_k)} (\z'_1,\ldots \z'_m) \otimes  V_{(N_{k+1},\ldots, N_r)} (\z''_1,\ldots \z''_n) \]
is a Kronecker tensor product (\S \ref{KroneckerCase}). Thus 
\[ \det( V_{(N_1,\ldots, N_r)} (\z_k )  )=  \det( V_{(N_1,\ldots, N_k)} (\z'_i))^{n}  \times \det( V_{(N_{k+1},\ldots, N_r)} (\z''_j))^m  \ .  \]
Now let $N_i = w_i N$ be  integers.   Since the points $\z'_i$  and $\z''_j$ may be chosen independently, we may take the limit as $N\rightarrow \infty$ to conclude that
\[   t_{\underline{w}} (K) \geq  t_{\underline{w}'} (K_1)^{|\underline{w}'|/|\underline{w}|} \times t_{{\underline{w}''}}(K_2)^{|\underline{w}''|/|\underline{w}|}\ .  \]
\end{proof}


\bibliographystyle{alpha}

\bibliography{biblio}

  \end{document}